\documentclass[12pt,a4paper]{amsart}
\usepackage{amssymb,amscd}

\theoremstyle{plain}
\newtheorem{theorem}{Theorem}[section]
\newtheorem{lemma}[theorem]{Lemma}
\newtheorem{proposition}[theorem]{Proposition}

\theoremstyle{definition}
\newtheorem{definition}[theorem]{Definition}
\newtheorem{remark}[theorem]{Remark}
\newtheorem{remarks}[theorem]{Remarks}

\numberwithin{equation}{section}

\newcommand\bA{{\mathbb A}}

\newcommand\bF{{\mathbb F}}
\newcommand\bG{{\mathbb G}}

\newcommand\bP{{\mathbb P}}
\newcommand\bQ{{\mathbb Q}}

\newcommand\bZ{{\mathbb Z}}

\newcommand\charact{\operatorname{char}}

\newcommand\id{\operatorname{id}}
\newcommand\lcm{\operatorname{lcm}}

\newcommand\red{\operatorname{red}}

\newcommand\rk{\operatorname{rk}}

\newcommand\GL{\operatorname{GL}}
\newcommand\Hom{\operatorname{Hom}}
\newcommand\Ima{\operatorname{Im}}

\newcommand\Ker{\operatorname{Ker}}

\newcommand\SL{\operatorname{SL}}

\newcommand\Spec{\operatorname{Spec}}
\newcommand\Tr{\operatorname{Tr}}

\title[Homogeneous varieties over finite fields]
{Counting points of homogeneous varieties over finite fields}

\author{Michel Brion and Emmanuel Peyre}
\address{Universit\'e de Grenoble I\\
D\'epartement de Math\'ematiques\\
Institut Fourier, UMR 5582 du CNRS\\
38402 Saint-Martin d'H\`eres Cedex, France}
\email{Michel.Brion@ujf-grenoble.fr}
\email{Emmanuel.Peyre@ujf-grenoble.fr}

\begin{document}
 
\begin{abstract}
Let $X$ be an algebraic variety over a finite field $\bF_q$,
homogeneous under a linear algebraic group. We show that there exists 
an integer $N$ such that for any positive integer $n$ in a fixed residue 
class mod $N$, the number of rational points of $X$ over $\bF_{q^n}$ is 
a polynomial function of $q^n$ with integer coefficients. 
Moreover, the shifted polynomials, where $q^n$ is formally replaced 
with $q^n + 1$, have non-negative coefficients.
\end{abstract}

\maketitle

\section{Introduction and statement of the results}
\label{sec:introduction}

Given an algebraic variety $X$ over a finite field $k = \bF_q$,
one may consider the points of $X$ which are rational over an
arbitrary finite field extension $\bF_{q^n}$. The number of these
points is given by Grothendieck's trace formula,
\begin{equation}\label{eqn:tra}
\vert X(\bF_{q^n}) \vert = 
\sum_{i \geq 0} (-1)^i\Tr \big(F^n,H_c^i(X)\big),
\end{equation}
where $F$ denotes the Frobenius endomorphism of $X_{\bar{k}}$ and
$H^i_c(X)$ stands for the $i$th $\ell$-adic cohomology group of
$X_{\bar{k}}$ with proper supports, $\ell$ being a prime not dividing 
$q$  (see e.g. \cite[Thm.~3.2,p.~86]{De77}). Moreover, 
by celebrated results of Deligne (see \cite{De74, De80}), each
eigenvalue $\alpha$ of $F$ acting on $H_c^i(X)$ is an algebraic
number, and all the complex conjugates of $\alpha$ have absolute value
$q^{\frac{w}{2}}$ for some non-negative integer  $w \leq i$, with
equality if $X$ is smooth and complete. This implies the general
properties of the counting function  
$n \mapsto \vert X(\bF_{q^n}) \vert$ 
predicted by the Weil conjectures.

We shall obtain more specific properties of that function 
under the assumption that $X$ is \emph{homogeneous}, i.e., 
admits an action of an algebraic group $G$ over $k$ such that 
$X(\bar{k})$ is a unique orbit of $G(\bar{k})$; then $X$ is of course
smooth, but possibly non-complete. We begin with a structure result
for these varieties:

\begin{theorem}\label{thm:prod}
Let $X$ be a homogeneous variety over a finite field $k$. Then
\begin{equation}\label{eqn:iso}
X \cong (A \times Y)/\Gamma,
\end{equation}
where $A$ is an abelian $k$-variety, $Y$ is a homogeneous $k$-variety 
under a connected linear algebraic $k$-group $H$, and $\Gamma$ is a
finite commutative $k$-group scheme which acts faithfully on $A$ by 
translations, and acts faithfully on $Y$ by automorphisms commuting 
with the action of $H$.

Moreover, $A$, $Y$ and $\Gamma$ are unique up to compatible
isomorphisms, $Y/\Gamma$ is a homogeneous $k$-variety under $H$, and
there is a canonical isomorphism
\begin{equation}\label{eqn:ten}
H^*_c(X) \cong H^*(A) \otimes H^*_c(Y/\Gamma).
\end{equation}

In particular,
\begin{equation}\label{eqn:pro}
\vert X(\bF_{q^n}) \vert = 
\vert A(\bF_{q^n}) \vert \; \; \vert (Y/\Gamma)(\bF_{q^n}) \vert.
\end{equation}
\end{theorem}

Theorem \ref{thm:prod} is deduced in Section \ref{sec:prod} from a
structure result for algebraic groups over finite fields, due to Arima 
(see \cite{Ar60}). 

In view of (\ref{eqn:pro}) and the known results on the counting
function of abelian varieties, we may concentrate on homogeneous
varieties under linear algebraic groups. For these, we obtain:

\begin{theorem}\label{thm:pepo}
Let $X$ be a variety over $\bF_q$, homogeneous under a linear
algebraic group. Then $\vert X(\bF_{q^n}) \vert$
is a periodic polynomial fonction of $q^n$ with integer coefficients.
\end{theorem}

By this, we mean that there exist a positive integer $N$ and 
polynomials $P_0(t),\ldots,P_{N - 1}(t)$ in $\bZ [t]$
such that
\begin{equation}\label{eqn:pepo}
\vert X(\bF_{q^n}) \vert = P_r(q^n) \quad \text{whenever}
\quad n \equiv r \quad (\text{mod}~N).
\end{equation}
We then say that $N$ is a \emph{period} of the function 
$q^n \mapsto \vert X(\bF_{q^n}) \vert$.

Notice that $\vert X(\bF_{q^n}) \vert$ is generally not a polynomial
function of $q^n$. For example, if $\charact(k)\neq 2$, then the affine
conic $X \subset \bA^2_k$ with equation $x^2 -a y^2 = b$ is
homogeneous under the corresponding orthogonal group and satisfies
$\vert X(\bF_{q^n}) \vert = q^n - \varepsilon$, where  
$\varepsilon = 1$ if $a$ is a square in $\bF_{q^n}$, and
$\varepsilon = - 1$ otherwise.

Theorem \ref{thm:pepo} is proved in Section \ref{sec:pepo},
by showing that each eigenvalue of $F$ acting on $H^*_c(X)$ is the
product of a non-negative integer power of $q$ with a root of unity
(Proposition \ref{prop:eig}). As a consequence, there exists a 
unique polynomial $P_X(t) \in \bZ[t]$ such that
\begin{equation}\label{eqn:pol}
P_X(q^n) = \vert X(\bF_{q^n}) \vert
\end{equation}
for any sufficiently divisible, positive integer $n$. 
Our third result yields a factorization of that polynomial:

\begin{theorem}\label{thm:fact}
Let $X$ be a variety over $\bF_q$, homogeneous under a linear
algebraic group, and let $P_X(t)$ be the polynomial satisfying
(\ref{eqn:pol}). Then there exists a non-negative integer $r$ such that
\begin{equation}\label{eqn:fac}
P_X(t) = (t-1)^r Q_X(t),
\end{equation}
where $Q_X(t)$ is a polynomial with non-negative integer coefficients.
\end{theorem}

This result follows from \cite[Thm.~1]{BP02} when $X$ is obtained from
a complex homogeneous variety by reduction modulo a large prime. 
However, certain homogeneous varieties over finite fields do not 
admit any lift to varieties in characteristic zero (see \cite{LR97}
for specific examples). Also, the approach of \cite{BP02} relies on
the existence of Levi subgroups, which fails in our setting,
and on arguments of equivariant cohomology which would require
non-trivial modifications.

We present a proof of Theorem \ref{thm:fact} in Section
\ref{sec:fact}; it combines the reduction steps of Section
\ref{sec:pepo} with a result adapted from \cite{BP02} in a simplified
form (Lemma \ref{lem:tor}, the only ingredient which relies 
on methods of $\ell$-adic cohomology). 

In Section \ref{sec:elem}, we show how to replace this ingredient with 
arguments of invariant theory, along the lines of classical results of
Steinberg (see \cite[\S 14]{St68}). This yields elementary proofs of
Theorems \ref{thm:pepo} and \ref{thm:fact}, and also of our most
surprising result:

\begin{theorem}\label{thm:pos}
Let $X$ be a variety over $\bF_q$, homogeneous under a linear
algebraic group, and let $P_0(t), \ldots, P_{N-1}(t)$ be the 
polynomials satisfying (\ref{eqn:pepo}). Then the shifted polynomials 
$P_0(t + 1), \ldots, P_{N-1}(t + 1)$ have non-negative coefficients. 
\end{theorem}

A similar positivity result has been conjectured by Mozgovoy and Reineke 
in the setting of quiver moduli (see \cite[Rem.~6.5]{MR07}
and also \cite[Conj.~8.5]{Re08}). They also observed that the
existence of a decomposition of the considered moduli spaces into
locally closed tori would yield a geometric explanation for 
their positivity property. Note that such a decomposition generally
does not exist in the setting of homogenenous varieties under linear
algebraic groups, since some of these varieties are not rational over
$\bar{k}$ (this follows from results of Saltman, see 
\cite{Sa84a,Sa84b}). This raises the question of finding a (geometric 
or combinatorial) interpretation of the coefficients of our shifted 
polynomials.

\medskip

\noindent
{\bf Acknowledgements.} We thank J.-P.~Serre and D.~Timashev for their
interest in our results, and for useful suggestions. Also, we thank the
referee for his careful reading and valuable comments.

\medskip

\noindent
{\bf Notation and conventions.} Throughout this article, we fix a
finite field $k$ of characteristic $p$, with $q$ elements. Also, we
fix an algebraic closure $\bar{k}$ of $k$. For any positive integer
$n$, we denote by $\bF_{q^n}$ the unique subfield of $\bar{k}$
with $q^n$ elements; in particular, $k = \bF_q$.

By a \emph{variety}, we mean a geometrically integral, separated
scheme of finite type over $k$; morphisms (resp.~products) of
varieties are understood to be over $k$. An \emph{algebraic group} $G$
is a smooth group scheme of finite type over $k$; then each connected
component of $G$ is a variety. The identity element of $G$ is denoted
by $e_G$. Notice that every algebraic subgroup of $G$ is 
``defined over $k$'' with our conventions.

For any variety $X$, we set 
$$
X_{\bF_{q^n}} := X \times_k \bF_{q^n}, \quad
X_{\bar{k}} := X \times_k \bar{k},
$$
and we denote by $F$ the Frobenius endomorphism of $X_{\bar{k}}$. 

Given a prime number $\ell \neq p$, we set for simplicity 
$$
H^i(X) := H^i(X_{\bar{k}} ; \bQ_{\ell}),
$$
the $i$th $\ell$-adic cohomology group of $X_{\bar{k}}$. Our notation
for cohomology with proper supports is
$$
H^i_c(X) := H^i_c(X_{\bar{k}} ; \bQ_{\ell}).
$$
We shall use \cite{DG70,Sp98} as general references for
algebraic groups, and \cite{De77, Mi80} for \'etale cohomology.

\section{Proof of Theorem \ref{thm:prod}}
\label{sec:prod}

We may choose a connected algebraic group $G$ such that $X$ is 
homogeneous under $G$. By \cite[Thm.~1]{Ar60} (see also 
\cite[Thm.~4]{Ro61}), we have $G = A H$, where $A$ is the
largest abelian subvariety of $G$, and $H$ is the largest connected
linear algebraic subgroup of $G$; moreover, $A$ and $H$ centralize each 
other. So $G \cong (A \times H)/ (A \cap H)$, and we may assume that 
$$
G = A \times H. 
$$
Replacing $A$ and $H$ with quotient groups, we may also assume that
they both act faithfully on $X$.

Let $G_X$ denote the kernel of the $G$-action on $X$. Then $G_X$ is 
isomorphic to a subgroup of $A$ (via the first projection) and also to
a subgroup of $H$. Since $A$ is complete and $H$ is affine, it follows
that $G_X$ is finite.

Also, $X$ contains a $k$-rational point $x$ by Lang's theorem
(see \cite[Thm.~2]{La56}). Denote by $G_x$ its isotropy
subgroup-scheme; then $G_x$ is linear by the finiteness of $G_X$
together with \cite[Lem.~p.~154]{Ma63}. 
In particular, the reduced neutral component $K$ of $G_x$ (a closed
normal subgroup of $G_x$) is contained in $H$. Let $\Gamma := G_x/K$;
then $\Gamma$ is a finite group scheme acting on $G/K$ on the right
via the action of $G_x$ on $G$ by right multiplication, and
$$
X \cong G/G_x \cong (G/K)/\Gamma \cong (A \times Y)/\Gamma,
$$
where $Y := H/K$. Denoting by $N_G(K)$ the normalizer of $K$ in $G$,
we have
$$
\Gamma \subset N_G(K)/K = A \times N_H(K)/K.
$$
Let $\Gamma'$ denote the kernel of the projection of $\Gamma$ to
$A$. Then $\Gamma'$ (resp.~$\Gamma/\Gamma'$) is isomorphic to a
subgroup scheme of $N_H(K)/K$ (resp.~of $A$), 
and 
$$
X \cong \big(A \times (Y/\Gamma')\big)/(\Gamma/\Gamma').
$$
Thus, we may assume that $\Gamma$ acts faithfully on $A$ by
translations. On the other hand, $\Gamma$ acts $H$-equivariantly on
$Y$ via the action of $N_H(K)/K$ on $H/K$ on the right, and the
kernel of this action is isomorphic to a subgroup scheme of $A$ which
acts trivially on $X$. Thus, $\Gamma$ acts faithfully on $Y$. This
completes the proof of (\ref{eqn:iso}). 

To show the uniqueness of $(A,Y,\Gamma)$, we begin with a general 
observation: 

\begin{lemma}
Let $X$ be a variety over an arbitrary field. Then there exists an 
abelian variety $A_X$ acting faithfully on $X$, such that any action 
of an abelian variety $A$ on $X$ arises from a unique homomorphism 
$A \to A_X$. Moreover, $A_X$ centralizes any connected algebraic group
of automorphisms of $X$. 
\end{lemma}

\begin{proof}
Consider an abelian variety $A$ and a connected algebraic group $G$,
both acting faithfully on $X$. Then the morphism 
$$
f: A \times G \times X \longrightarrow X, \quad 
(a, g, x) \longmapsto a g a^{-1} g^{-1} x
$$
satisfies $f(a, e_G, x) = x$. By the rigidity lemma of
\cite[p.~43]{Mu70}, it follows that $f$ factors through the projection
$p_{23} : A \times G \times X \to G \times X$. 
But $f(e_A,g,x) = x$, so that $f$ factors through the projection
$p_3 : A \times G \times X \to X$. In other words, $A$ centralizes $G$.

On the other hand, $A$ stabilizes the smooth locus $U$ of $X$. By a
theorem of Nishi and Matsumura, the induced action of $A$ on the
Albanese variety of $U$ has a finite kernel (see \cite{Ma63}, or
\cite[Thm.~2]{Br07} for a more modern version). In particular, 
$\dim(A) \leq \dim(U) = \dim(X)$. 

Combining these two steps yields our statement. 
\end{proof}

\begin{remark}
For a variety $X$, there may exist an infinite sequence 
$G_1 \subset G_2 \subset \cdots G_n \subset \cdots$
of closed connected algebraic groups, all acting faithfully and
transitively on $X$. This happens e.g. for the variety
$X = (\bA^1 \setminus \{ 0\}) \times \bA^1$ 
and the group $G_n$ consisting of automorphisms
$$ 
x \longmapsto a x, \quad y \longmapsto y + P(x),
$$
where $a \in \bG_m$ and $P$ is a polynomial of degree $\leq n$.
\end{remark}

Returning to the situation of (\ref{eqn:iso}), we claim that 
$A_X = A$. To see this, consider the action of $A$ on $X$ via its
action on itself by translations. The projection 
$p_2 : A \times Y \to Y$ induces a morphism
\begin{equation}\label{eqn:tor}
p: X \to Y/\Gamma
\end{equation}
which is an $A$-torsor for the fppf topology (since the quotient 
morphism $Y \to Y/\Gamma$ is a $\Gamma$-torsor, and hence the square
$$
\CD 
A \times Y @>{p_2}>> Y \\
@V{/\Gamma}VV @V{/\Gamma}VV \\
X @>{p}>> Y/\Gamma \\
\endCD
$$
is cartesian). Also, note that the quotient variety 
$$
Y/\Gamma = X/A = G/G_x A
$$ 
exists and is homogeneous under $H = G/A$. Thus, $A$ is contained in
$A_X$, and the quotient $A_X/A$ acts on $Y/\Gamma$. Since any 
morphism from the connected linear algebraic group $H$ to an abelian
variety is constant, the Albanese variety of $Y/\Gamma$ is trivial. 
By the Nishi-Matsumura theorem again, it follows that the action of
the abelian variety $A_X/A$ on $Y/\Gamma$ is trivial as well. In
particular, each $A_X(\bar{k})$-orbit in $X(\bar{k})$ is an
$A(\bar{k})$-orbit. This implies $\dim(A_X) = \dim(A)$, which proves
our claim.

As a consequence, $A$ (and $Y/\Gamma$) depend only on $X$. On the 
other hand, the natural map 
$$
q : X  = (A \times Y)/\Gamma \longrightarrow A/\Gamma
$$ 
is a morphism to an abelian variety, with fibers isomorphic to the 
homogeneous variety $Y$ under $H$. It follows that $q$ is the Albanese 
morphism of $X$. In particular, the subgroup scheme $\Gamma$ of $A$, 
and the $\Gamma$-variety $Y$, depend only on $X$. This shows the desired
uniqueness.

To prove the isomorphism (\ref{eqn:ten}), we first consider the case
where the group scheme $\Gamma$ is reduced. Then we have canonical
isomorphisms
$$
\displaylines{
H^*_c(X) \cong H^*_c(A\times Y)^{\Gamma}
\cong 
\big(H^*(A) \otimes H^*_c(Y)\big)^{\Gamma}  
\cr \hfill \cr 
\cong 
H^*(A) \otimes H^*_c(Y)^{\Gamma}
\cong
H^*(A) \otimes H^*_c(Y/\Gamma),
\cr}$$
where the first and last isomorphism follow from Lemma \ref{lem:fin}
(i) below, the second one from the K\"unneth isomorphism and
the properness of $A$, and the third one holds since the action of
$\Gamma$ on $H^*(A)$ is trivial (indeed, $\Gamma$ acts on $A$ by
translations).

In the general case, the reduced subscheme $\Gamma_{\red}$ is a finite 
subgroup of $\Gamma$, and the natural map
$$
(A \times Y)/\Gamma_{\red} \to (A \times Y)/\Gamma = X
$$
is finite and bijective on $\bar{k}$-rational points. By Lemma 
\ref{lem:fin} (ii) below, it follows that
$$
H^*_c(X) \cong H^*_c\big((A\times Y)/\Gamma_{\red}\big).
$$
Together with the preceding step and Lemma \ref{lem:fin} (ii) again, 
this yields the isomorphism (\ref{eqn:ten}).

Finally, (\ref{eqn:pro}) follows by combining (\ref{eqn:tra}) and 
(\ref{eqn:ten}) or, more directly, by considering the morphism 
(\ref{eqn:tor}): for any $z \in (Y/\Gamma)(\bF_{q^n})$, the fiber
$X_z$ (a variety over $\bF_{q^n}$) is a torsor under
$A_{\bF_{q^n}}$. By Lang's theorem, it follows that $X_z$ contains
$\bF_{q^n}$-rational points, and these form a unique orbit of
$A(\bF_{q^n})$.

\begin{lemma}\label{lem:fin}
{\rm (i)} Let $\Gamma$ be a finite group acting on a variety $X$ 
such that the quotient morphism $f : X \to Y$ exists, where $Y$ is
a variety (this assumption is satisfied if $X$ is quasi-projective, 
see \cite[p.~69]{Mu70}). Then $\Gamma$ acts on $H^*_c(X)$, and we have 
a canonical isomorphism
$$
H^*_c(Y) \cong H^*_c(X)^{\Gamma}.
$$ 

\noindent
{\rm (ii)} Let $f: X \to Y$ be a finite morphism of varieties,
bijective on $\bar{k}$-rational points. Then we have a canonical 
isomorphism
$$
H^*_c(Y) \cong H^*_c(X).
$$ 
\end{lemma}

\begin{proof}
(i) Note that $f_!\bQ_{\ell} = f_*\bQ_{\ell}$ and 
$R^i f_!\bQ_{\ell} = 0$ for all $i \geq 1$, since $f$ is finite. This
yields a canonical isomorphism
\begin{equation}\label{eqn:dir}
H^*_c(X) \cong 
H^*_c(Y_{\bar{k}} ; f_* \bQ_{\ell}).
\end{equation}
Moreover, $\Gamma$ acts on $f_* \bQ_{\ell}$ and hence on $H^*_c(X)$. 
Thus, (\ref{eqn:dir}) restricts to an isomorphism
$$
H^*_c(X)^{\Gamma} \cong H^*_c(Y_{\bar{k}} ; (f_* \bQ_{\ell})^{\Gamma}).
$$
To complete the proof, it suffices to show that the natural map from
the constant sheaf $\bQ_{\ell}$ to $f_* \bQ_{\ell}$ induces an
isomorphism $\bQ_{\ell} \cong (f_* \bQ_{\ell})^{\Gamma}$. In turn, it
suffices to prove that 
\begin{equation}\label{eqn:com}
H^0(X_{\bar{y}}; \bQ_{\ell})^{\Gamma} \cong \bQ_{\ell},
\end{equation}
where $X_{\bar{y}}$ denotes the geometric fiber of $f$ at an arbitrary
point $y \in Y$. But $X_{\bar{y}}$ is a finite scheme over the field
$\overline{\kappa(y)}$, equipped with an action of $\Gamma$ which
induces a transitive action on its set of connected components; this
implies (\ref{eqn:com}). 

(ii) is checked similarly; here the map 
$\bQ_{\ell} \to f_*\bQ_{\ell}$ is an  isomorphism. 
\end{proof}

\begin{remarks}
(i) Lemma \ref{lem:fin} is certainly well-known, but we could not
locate a specific reference. The first assertion is exactly  
\cite[(5.10)]{Sr79}; however, the proof given there is only valid for 
$\Gamma$-torsors.

\medskip

\noindent
(ii) If $X$ in Theorem \ref{thm:prod} is complete, then $\Gamma$ is 
trivial in view of a result of Sancho de Salas (see \cite{SS03}). 
Moreover, we have $Y \cong H/Q$, where $Q$ is a subgroup scheme of 
$H$ such that the reduced subscheme $Q_{\red}$ is a parabolic subgroup. 
It follows easily that $\vert Y(\bF_{q^n}) \vert$ is a polynomial 
function of $q^n$ (for details, see Steps 1 and 3 in Section 
\ref{sec:fact}).

For an arbitrary homogeneous variety $X$, the subgroup scheme $\Gamma$
is generally non-trivial. Indeed, consider an abelian variety $A$ having
a $k$-rational point $p$ of order $2$. Let also $Y := \SL(2)/T$, where
$T\subset \SL(2)$ denotes the diagonal torus. The group $\Gamma$ of
order $2$ acts on $A$ via translation by $p$, and on $Y$ via right
multiplication by the matrix
$\begin{pmatrix}
0  & -1 \\ 
1 & 0 \\
\end{pmatrix}$ 
which normalizes $T$; the variety $X := (A \times Y)/\Gamma$ is the
desired example. One easily checks that
$$
\vert (Y/\Gamma)(\bF_{q^n}) \vert = q^{2n}, \quad \text{whereas} \quad 
\vert Y(\bF_{q^n}) \vert = q^n (q^n + 1).
$$
Thus, $Y/\Gamma$ cannot be replaced with $Y$ in the equality
(\ref{eqn:pro}). 

\medskip

\noindent
(iii) The isomorphism (\ref{eqn:iso}) only holds for homogeneous
varieties defined over finite fields. Consider indeed a field $k$ 
which is not algebraic over a finite subfield.
By \cite{ST67}, there exists an elliptic curve $C$ over $k$,
having a $k$-rational point $x$ of infinite order. Let $L$ be the 
line bundle on $C$ associated with the divisor $(x) - (0)$. 
Denote by $G$ the complement of the zero section in the total space 
of $L$, and by $q : G \to C$ the projection; then $q$ is a torsor under
the multiplicative group $\bG_m$. In fact, $G$ has a structure of an 
algebraic group over $k$, extension of $C$ by $\bG_m$; in particular, 
$q$ is the Albanese map. If the isomorphism (\ref{eqn:iso}) holds for 
$G$, then $C \cong A/\Gamma$ and $Y \cong \bG_m$. Thus, $G$ has 
non-constant regular functions, namely, the non-constant regular 
functions on the quotient $Y/\Gamma \cong \bG_m$. In other words, 
there exists an integer $n \neq 0$ such that the power $L^n$ has a 
non-zero section; but this is impossible, since $L^n$ is a non-trivial 
line bundle of degree $0$.

The above group $G$ is an example of an anti-affine algebraic group
in the sense of \cite{Br09}. That article contains a classification 
of these groups, and further examples in characteristic zero.      
\end{remarks}

\section{Proof of Theorem \ref{thm:pepo}}
\label{sec:pepo}

First, it suffices to show that $\vert X(\bF_{q^n})\vert$
is a \emph{periodic Laurent polynomial function of $q^n$ with 
algebraic integer coefficients}, i.e., there exist a positive integer
$N$ and $P_0(t), \ldots, P_{N-1}(t) \in \bar{\bZ}[t,t^{-1}]$ satisfying
(\ref{eqn:pepo}), where $\bar{\bZ}$ denotes the ring of algebraic
integers. Indeed, if $P(t) \in \bar{\bZ}[t,t^{-1}]$ and 
$P(q^n)$ is an integer for infinitely many positive integers $n$, 
then $P(t) \in \bZ[t]$.  

Next, it suffices to show the following: 

\begin{proposition}\label{prop:eig}
Let $X$ be a homogeneous variety under a linear algebraic group. 
Then each eigenvalue $\alpha$ of $F$ acting on $H^*_c(X)$ is of the
form $\zeta \, q^j$, where $\zeta = \zeta(\alpha)$ is a root of 
unity, and $j = j(\alpha)$ is an integer.
\end{proposition}

Indeed, in view of Grothendieck's trace formula (\ref{eqn:tra}), 
that proposition implies readily that $\vert X(\bF_{q^n})\vert$ 
is a periodic Laurent polynomial function of $q^n$, with coefficients 
being sums of roots of unity.

Before proving the proposition, we introduce two notions 
which will also be used in the proof of Theorem \ref{thm:fact}.

\begin{definition} 
We say that a variety $X$ is \emph{weakly pure}, if it satisfies the
assertion of Proposition \ref{prop:eig}.

Also, $X$ is \emph{strongly pure} if $H^i_c(X) = 0$ for any odd $i$, 
and for any even $i$, each eigenvalue of $F$ acting on $H^i_c(X)$ is
of the form $\zeta \, q^{\frac{i}{2}}$, where $\zeta$ is a root of
unity.
\end{definition}

Clearly, any strongly pure variety $X$ is weakly pure; it is also pure
in the (usual) sense that all the complex conjugates of eigenvalues of
$F$ acting on $H^i_c(X)$ have absolute value $q^{\frac{i}{2}}$, for all
$i$. Yet some weakly pure varieties are not pure, e.g., tori.

Weak and strong purity are preserved under base change by any finite
extension; specifically, a variety $X$ is weakly (resp.~strongly) pure
if and only if so is $X_{\bF_{q^n}}$ for some (or for any) positive
integer $n$. Further easy properties of these notions are gathered in
the following: 

\begin{lemma}\label{lem:wep}
{\rm (i)} Let $\Gamma$ be a finite group acting on a variety $X$, 
such that the quotient variety $Y = X/\Gamma$ exists. If $X$ is 
weakly (resp.~strongly) pure, then so is $Y$.

\noindent
{\rm (ii)} Let $f : X \to Y$ be a finite morphism of varieties, 
bijective on $\bar{k}$-rational points. Then $X$ is weakly 
(resp.~strongly) pure if and only if so is $Y$.

\noindent
{\rm (iii)} Let $Y$ be a closed subvariety of a variety $X$, with
complement $U$. If both $Y$ and $U$ are weakly (resp.~strongly)
pure, then so is $X$.
\end{lemma}

\begin{proof}
(i) and (ii) follow from  Lemma \ref{lem:fin}, and (iii) from the
exact sequence 
$
H^i_c(U) \longrightarrow H^i_c(X) \longrightarrow H^i_c(Y)$.
\end{proof}

Next, we obtain a result of independent interest, which is the main
ingredient of the proof of Proposition \ref{prop:eig}:

\begin{proposition}\label{prop:tors}
Let $G$ be a connected linear algebraic group, and $\pi : X \to Y$ 
a $G$-torsor, where $X$ and $Y$ are varieties. Then $X$ is weakly 
(resp.~strongly) pure if and only if so is $Y$.
\end{proposition}

\begin{proof}
(i) We may choose a Borel subgroup $B$ of $G$ and a maximal torus 
$T$ of $B$. Then $\pi$ is the composite morphism
$$
\CD
X @>{\pi_T}>> X/T @>{\varphi}>> X/B @>{\psi}>> X/G = Y,
\endCD
$$
where $\pi_T$ is a $T$-torsor, $\varphi$ is smooth with fiber $B/T$
isomorphic to the unipotent radical of $B$, and $\psi$ is projective 
and smooth with fiber $G/B$, the flag variety of $G$.

We claim that $X/T$ is weakly (resp.~strongly) pure if and only if 
so it $X/B$. Indeed, $B/T$ is isomorphic to an affine space $\bA^d$, 
and hence $R^i \varphi_! \bQ_{\ell} = 0$ for all $i \neq 2d$, while
$R^{2d} \varphi_!\bQ_{\ell} \cong \bQ_{\ell}(-d)$ via the trace map.
This yields a canonical isomorphism
$$
H^i_c(X/B) \cong H^{i + 2d} _c(X/T)(d).
$$
Thus, the eigenvalues of $F$ in $H^i_c(X/B)$ are exactly the 
products $\beta \, q^{-d}$, where $\beta$ is an eigenvalue of $F$ in 
$H^{i + 2d}_c(X/T)$. This implies our claim.

Next, we claim that $X/B$ is weakly (resp.~strongly) pure if and only 
if so is $X/G$. Indeed, the Leray spectral sequence associated with
the flag bundle $\psi$ degenerates (since the cohomology ring of the
fiber $G/B$ is generated by Chern classes of line bundles associated
with characters of $B$, and all such line bundles extend to $X/B$);
moreover, the sheaves 
$R^j \psi_! \bQ_{\ell} = R^j \psi_* \bQ_{\ell}$ are constant. This
yields an isomorphism of graded $\bQ_{\ell}$-vector spaces with
$F$-action
$$
H^*_c(X/B) \cong H^*_c(X/G) \otimes H^*(G/B).
$$
In particular, $H^*_c(X/G)$ may be identified with a $F$-stable
subspace of $H^*_c(X/B)$. Thus, if $X/B$ is weakly (resp.~strongly)
pure, then so is $X/G$. The converse holds since $G/B$ is strongly
pure (as follows from the Bruhat decomposition, see Step 3 in Section 
\ref{sec:fact} for details).

By combining both claims, we may assume that $G = T$. Replacing $k$
with a finite extension, we may further assume that $T$ is
split. Thus, we are reduced to the case where $G = T = \bG_m$. Then we
have the Gysin long exact sequence
$$
\CD
\cdots  
H^i_c(X) @>>> H^{i-2}_c(Y)(2) @>{c_1(L)}>> H^i_c(Y) @>>> H^{i+1}_c(X) 
\cdots,
\endCD 
$$
where $c_1(L)$ denotes the multiplication by the first Chern class of 
the invertible sheaf $L$ associated with the $\bG_m$-torsor 
$\pi: X \to Y$. Thus, if $Y$ is weakly (resp.~strongly) pure, then so 
is $X$. The converse is obtained by decreasing induction on $i$, 
since $H^i_c(Y) = 0$ for each $i > 2 \dim(Y)$, and $F$ acts on 
$H^{2\dim(Y)}_c(Y)$ via multiplication by $q^{\dim(Y)}$. 
\end{proof}

We may now prove Proposition \ref{prop:eig}. We have $X \cong G/H$, 
where $G$ is a linear algebraic group, and $H$ a closed subgroup 
scheme. Since $X$ is a variety, we may assume that $G$ is connected.
Moreover, since the reduced subscheme $H_{\red}$ is a closed 
algebraic subgroup, and the natural map 
$G/H_{\red} \to G/H$ is finite and bijective on $\bar{k}$-rational 
points, we may assume that $H$ is an algebraic group in view of 
Lemma \ref{lem:fin}(ii). 

Applying Proposition \ref{prop:tors} to the torsors $G \to \Spec(k)$
and $G \to G/H^0$ (where $H^0$ denotes the neutral component of $H$),
we see that $G$ and $G/H^0$ are weakly pure. Thus, so is
$G/H \cong (G/H^0)/(H/H^0)$, by Lemma \ref{lem:wep}(i).

\section{Proof of Theorem \ref{thm:fact}}
\label{sec:fact}

As above, we consider a homogeneous variety $X = G/H$, where $G$ 
is a connected linear algebraic group, and $H$ is a closed 
subgroup scheme. We first reduce to the case where $G$ is reductive, 
and $H$ is a closed subgroup such that $H^0$ is a torus. For this, 
we carry out a sequence of four reduction steps, where we use
elementary counting arguments rather than $\ell$-adic cohomology and
the Grothendieck trace formula, to prepare the way for the completely
elementary proofs of Section \ref{sec:elem}. 

\medskip

\noindent
\emph{Step 1.} 
Since the natural map $(G/H_{\red})(\bar{k}) \to (G/H)(\bar{k})$ is
bijective, we have
$\vert (G/H)(\bF_{q^n}) \vert = \vert (G/H_{\red})(\bF_{q^n}) \vert$
for any positive integer $n$, and hence 
$$
P_{G/H}(t) = P_{G/H_{\red}}(t).
$$ 
Replacing $(G,H)$ with $(G,H_{\red})$, we may thus assume that 
\emph{$H$ is an algebraic group}.

\medskip

\noindent
\emph{Step 2.} 
The unipotent radical $R_u(G)$ acts on $X$, with quotient morphism 
the natural map $f: G/H \to G/R_u(G) H$. 
The fiber of $f$ at any coset $g R_u(G) H$ equals
$$
g R_u(G) H/H \cong R_u(G)/\big(R_u(G) \cap g H g^{-1} \big) \cong 
R_u(G)/g \big(R_u(G) \cap H \big)g^{-1}.
$$
The induced map
$(G/H)(\bF_{q^n}) \to \big(G/R_u(G) H)(\bF_{q^n}\big)$
is surjective by Lang's theorem. Moreover, since $R_u(G)$ is connected
and unipotent, each fiber has $q^{nd}$ elements, where 
$d := \dim R_u(G)/\big(R_u(G) \cap H \big)$. Hence
$$
P_{G/H}(t) = t^{d} P_{G/R_u(G)H}(t).
$$
Replacing $G$ with the quotient $G/R_u(G)$ and $H$ with its image in
that quotient, we may assume in addition that $G$ \emph{is reductive}. 

\medskip

\noindent
\emph{Step 3.}
If $H$ is not reductive, then it is contained in some proper
parabolic subgroup $P$ of $G$ (see \cite{BT71}). We may further
assume that $H$ is not contained in a proper parabolic subgroup
of $P$; then the image of $H$ in the reductive group $P/R_u(P)$ 
is reductive as well.
 
Choose a Borel subgroup $B$ of $P$ and a maximal torus $T$ 
of $B$; denote by $U$ the unipotent radical of $B$, and by $N_G(T)$ 
the normalizer of $T$ in $G$. Then 
$W := N_G(T)/T$ is the Weyl group of $(G,T)$; it contains the Weyl group 
$W_P$ of $(P,T)$. Choose a set $W^P$ of representatives in $W$ of the
quotient $W/W_P$. Also, for any $w \in W^P$, choose a representative
$n_w \in N_G(T)$. Then, by the Bruhat decomposition, we have
\begin{equation}\label{eqn:bru}
G  = \bigcup_{w \in W^P} B n_w P,
\end{equation}
where the $B n_w P$ are disjoint locally closed subvarieties. 
Moreover, for any $w \in W^P$, there exists a closed subgroup
$U^w$ of $U$, normalized by $T$, such that the map
\begin{equation}\label{eqn:cel}
U^w \times P \longrightarrow  B n_w P, \quad
(x,y) \longmapsto x n_w y
\end{equation}
is an isomorphism; each $U^w$ is isomorphic to an affine space. Thus,
$G/H$ is the disjoint union of the locally closed subvarieties
$$
B n_w P/H \cong U^w \times P/H.
$$  
It follows that
$$
\vert (G/H)(\bF_{q^n}) \vert = \vert (G/P)(\bF_{q^n}) \vert \; 
\vert (P/H)(\bF_{q^n}) \vert,
$$
and $\vert (G/P)(\bF_{q^n}) \vert$ is a polynomial function of $q^n$
with non-negative integer coefficients. This yields the factorization 
$$
P_{G/H}(t) = P_{G/P}(t) \; P_{P/H}(t).
$$
Thus, we may replace $(G,H)$ with $(P,H)$. By Step 2, we may further
replace $(P,H)$ with $\big(P/R_u(P), R_u(P)H/R_u(P)\big)$. So we may 
assume that \emph{$G$ and $H$ are both reductive}.

\medskip

\noindent
\emph{Step 4.}
We now choose a maximal torus $T_H$ of $H$, and denote by $N_H$ its 
normalizer in $H$. We claim that the natural map 
$$
p : G/N_H \to G/H
$$ 
induces a surjective map
$(G/N_H)(\bF_{q^n}) \to (G/H)(\bF_{q^n})$ 
such that each fiber has $q^{n\dim(H/N_H)}$ elements.

Indeed, consider $g \in G(\bar{k})$ such that the coset $g H_{\bar{k}}$
lies in $(G/H)(\bF_{q^n})$. Then the subgroup 
$(g H g^{-1})_{\bar{k}}$ of $G_{\bar{k}}$ is 
defined over $\bF_{q^n}$. Moreover, the fiber of $p$ at 
$g H_{\bar{k}}$ is isomorphic to the variety of maximal tori of 
$(g H g^{-1})_{\bar{k}}$; hence its number of $\bF_{q^n}$-rational 
points equals $q^{n \dim(H/N_H)}$, by a theorem of Steinberg 
(see \cite[Cor.~14.16]{St68}). This implies our claim and, in turn,
the equality
$$
P_{G/H}(t) = t^{-\dim(H/N_H)} \; P_{G/N_H}(t).
$$
Replacing $(G,H)$ with $(G,N_H)$, we may thus assume that $H^0$ 
\emph{is a torus}. 

\medskip

We may also freely replace $k$ with any finite extension,
since this does not affect the definition of $P_X$, nor the statement
of Theorem \ref{thm:fact}.

By Lemma \ref{lem:tor} below (a version of 
\cite[Lem.~1, Lem.~2]{BP02}), there exists a torus $S \subset G$ 
acting on $G/H^0$ with finite isotropy subgroup schemes, such that
the quotient $S \backslash G/H^0$ is strongly pure. Then $S$ also
acts on $X = G/H$ with finite isotropy subgroup schemes, and
the quotient 
$$
S \backslash X = (S \backslash G/H^0)/(H/H^0)
$$ 
is also strongly pure by Lemma \ref{lem:wep}(i). 

We now claim that there exists a decomposition of $X$ into finitely 
many locally closed $S$-stable subvarieties 
$$
X_i \cong (S/\Gamma_i) \times Y_i,
$$ 
where $\Gamma_i$ is a finite subgroup scheme of $S$. Indeed, by a 
theorem of Chevalley, $X$ is $G$-equivariantly isomorphic to
an orbit in the projectivization $\bP(V)$ of a finite-dimensional 
$G$-module $V$. Choosing a basis of $S$-eigenvectors in the dual module 
$V^*$ yields homogeneous coordinates on $\bP(V)$, and hence a 
decomposition of $\bP(V)$ into locally closed $S$-stable tori $S_i$ 
(where some prescribed homogenous coordinates are non-zero, and all 
others are zero). 
Clearly, $S$ acts on each $S_i$ via a homomorphism $f_i : S \to S_i$.
Denote by $\Gamma_i$ the kernel of $f_i$. Then we may identify 
$S/\Gamma_i$ with a subtorus of $S_i$, and hence there exists a 
``complementary'' subtorus $S'_i \subset S_i$ such that the 
multiplication map $(S/\Gamma_i) \times S'_i \to S_i$ is an isomorphism.
If $S_i$ meets $X$, then $\Gamma_i$ is finite, and the natural map
$(S/\Gamma_i) \times (X \cap S'_i) \to X \cap S_i$ is an isomorphism.
Thus, our claim holds for the subvarieties $X_i := X \cap S_i$ and
$Y_i := X \cap S'_i$. 

That claim yields a similar decomposition of $S \backslash X$ into the 
subvarieties $Y_i$. Note that each $S/\Gamma_i$ is a torus of dimension  
$$
r := \dim(S) = \dim(T) - \dim(H).
$$
Thus, we have for any sufficiently divisible $n$:
$$
\displaylines{
\vert X(\bF_{q^n}) \vert = \sum_i \vert X_i(\bF_{q^n}) \vert 
= \sum_i \vert (S/\Gamma_i)(\bF_{q^n}) \vert \; \vert Y_i(\bF_{q^n}) \vert
\hfill \cr\hfill
= (q^n - 1)^r \sum_i \vert Y_i(\bF_{q^n}) \vert 
= (q^n - 1)^r \; \vert (S \backslash X)(\bF_{q^n}) \vert.
}$$
In other words, $P_X(t) = (t - 1)^r P_{S \backslash X}(t)$.
By strong purity, the coefficients of the polynomial 
$P_{S \backslash  X}(t)$ are non-negative; this yields the  
desired factorization.

\begin{lemma}\label{lem:tor}
Let $G$ be a connected reductive group, and $H \subset G$ a torus. 
Choose a maximal torus $T$ of $G$ containing $H$ and denote by 
$W$ the Weyl group of $(G,T)$. 

Then, possibly after base change by a finite extension $\bF_{q^N}$,
there exist subtori $S$ of $T$ such that $T = S \; w(H)$ and 
$S\cap w(H)$ is finite for all $w \in W$. 

Any such torus $S$ acts on $G/H$ with finite isotropy subgroup
schemes, and the quotient $S \backslash G/H$ is a strongly pure,
affine variety.
\end{lemma}

\begin{proof}
We may assume that $T$ is split. Let $\Lambda$ be its character group;
this is a lattice equipped with an action of $W$, and containing the
lattice $\Lambda^H$ of characters of $T/H$. We may find a subgroup
$\Lambda'$ of $\Lambda$ such that $\Lambda/\Lambda'$ is a lattice, 
$\Lambda' \cap w(\Lambda^H) = \{0\}$ and $\Lambda' + w(\Lambda^H)$ has
finite index in $\Lambda$ for any $w \in W$ (indeed, the subspaces
$w(\Lambda^H)_{\bQ}$ of the rational vector space $\Lambda_{\bQ}$ have
a common complement). Then $\Lambda' = \Lambda^S$ for a subtorus $S$
of $T$ which satisfies the first assertion.

For the second assertion, we may choose a Borel subgroup $B$ of $G$
that contains $T$. As in (\ref{eqn:bru}, \ref{eqn:cel}), consider the
Bruhat decomposition $G = \bigcup_{w \in W} B n_w B$ and the
isomorphisms
$$
U^w \times U \times T \longrightarrow B n_w B, \quad 
(u,v,t) \longmapsto u n_w v t. 
$$
The resulting projection 
$$
p: B n_w B \longrightarrow T
$$
is equivariant with respect to $T \times T$ acting on $B n_w B$ via
left and right multiplication, and on $T$ via
$$
(x,y) \cdot z = w^{-1}(x) \; z \; y^{-1}.
$$
The fiber of $p$ at the identity element of $T$ is isomorphic to 
the affine space $U^w \times U$. This yields a cartesian square 
$$ 
\CD
B n_w B @>{p}>> T \\
@V{/H}VV @V{/H}VV \\
B n_w B/H @>{f}>> T/H, \\
\endCD
$$
where $f$ is $T$-equivariant. Moreover, $T/H$ is homogeneous under $S$
acting via $s \cdot t H = w^{-1}(s) t H$, and the corresponding
isotropy subgroup scheme is $S \cap w^{-1}(H)$. Thus, $B n_w B/H$ is
the quotient of $U^w \times U \times S$ by $S \cap w^{-1}(H)$ acting
linearly on $U^w \times U$, and on $S$ via multiplication. 

By our assumption on $S$, it follows that all isotropy subgroup
schemes for its action on $G/H$ are finite; in particular,
all orbits are closed. Since the variety $G/H$ is affine, the quotient 
$S \backslash G/H$ exists and is affine as well. Moreover, this
quotient is decomposed into the locally closed varieties
$$
S \backslash B n_w B/H \cong (U \times U^w)/\big(S \cap w(H)\big).
$$
Thus, $S \backslash G/H$ is strongly pure in view of Lemma 
\ref{lem:wep}(i).
\end{proof}

\section{Elementary proofs of Theorems \ref{thm:pepo}, 
\ref{thm:fact} and \ref{thm:pos}}
\label{sec:elem}

As in Section \ref{sec:fact}, we may assume that $X = G/H$, where $G$ 
is a connected reductive group and $H$ is a closed subgroup such that
$H^0$ is a torus. We first obtain a formula for the number of
$\bF_{q^n}$-rational points of $X$, by standard arguments of Galois
descent.

Denote by $\Gamma$ the finite group $H/H^0$. For any 
$\gamma \in \Gamma$, choose a representative 
$h_{\gamma} \in H(\bar{k})$. By Lang's theorem, we may choose
$g_{\gamma} \in G(\bar{k})$ such that 
$h_{\gamma} = g_{\gamma}^{-1} F(g_{\gamma})$.

Consider a point $x \in (G/H)(\bar{k})$ with representative
$g \in G(\bar{k})$. Then $x \in (G/H)(\bF_q) = (G/H)^F$ if and only if
$g^{-1} F(g) \in H(\bar{k})$, that is, 
$g^{-1} F(g) \in h_{\gamma} H^0(\bar{k})$ for a unique 
$\gamma \in \Gamma$. Equivalently, we have $g = z g_{\gamma}$,
where $z \in G(\bar{k})$ satisfies 
$$
z^{-1} F(z) \in F(g_{\gamma}) H^0(\bar{k}) F(g_{\gamma}^{-1}).
$$
Moreover,
$$
F(g_{\gamma}) H^0 F(g_{\gamma}^{-1}) = 
F(g_{\gamma} H^0 g_{\gamma}^{-1}) = 
g_{\gamma} h_{\gamma} H^0 h_{\gamma}^{-1} g_{\gamma}^{-1} = 
g_{\gamma} H^0 g_{\gamma}^{-1}.
$$ 
Thus, each $g_{\gamma} H^0 g_{\gamma}^{-1}$ is defined over $k$, and 
$$
z^{-1} F(z) \in g_{\gamma} H^0(\bar{k}) g_{\gamma}^{-1}.
$$
Applying Lang's theorem again, we see that
$$
z \in G(\bF_q) g_{\gamma} H^0(\bar{k}) g_{\gamma}^{-1}.
$$ 
Thus, the preimage of $(G/H)(\bF_q)$ in $(G/H^0)(\bar{k})$ is the 
disjoint union of the orbits $G(\bF_q) g_{\gamma} H^0$, where 
$\gamma \in \Gamma$. This yields the equality
\begin{equation}\label{eqn:co0}
\vert (G/H)(\bF_q) \vert = 
\frac{1}{\vert \Gamma \vert} \sum_{\gamma \in \Gamma} 
\frac{\vert G(\bF_q) \vert}
{\vert (g_{\gamma} H^0 g_{\gamma}^{-1})(\bF_q) \vert}.
\end{equation}
But $G(\bF_q) = G^F$ and 
$(g_{\gamma} H^0 g_{\gamma}^{-1})(\bF_q) 
= (g_{\gamma} H^0 g_{\gamma}^{-1})^F \cong (H^0)^{\gamma^{-1} F}$,
where $\gamma$ acts on $H^0$ via conjugation by $h_{\gamma}$
(this makes sense since $H^0$ is commutative). Thus, we may rewrite
(\ref{eqn:co0}) as 
\begin{equation}\label{eqn:co1}
\vert (G/H)(\bF_q) \vert = \frac{1}{\vert \Gamma \vert} 
\sum_{\gamma \in \Gamma} 
\frac{\vert G^F \vert}{\vert (H^0)^{\gamma  F} \vert}.
\end{equation}
This still holds when $q$ is replaced with $q^n$, and $F$ with $F^n$,
where $n$ is an arbitrary positive integer.

Next, we obtain a more combinatorial formula for 
$\vert (G/H)(\bF_{q^n}) \vert$. Denote by 
$$
\Lambda = \Lambda_{H^0} := \Hom(H^0,\bG_m)
$$ 
the character group of the torus $H^0$. Then $\Lambda$ is a lattice,
where $\Gamma$ acts via its action on $H^0$ by conjugation. Moreover,
$F$ defines an endomorphism of $\Lambda$ that we still denote by $F$,
via $\big(F(\lambda)\big)(x) = \lambda\big(F(x)\big)$ for all points
$\lambda$ of $\Lambda$ and $x$ of $H^0$. Since $H^0$ splits over
some finite extension $\bF_{q^N}$, we have
$$
F^N = q^N \id
$$ 
as endomorphisms of $\Lambda$. Thus, we may write $F = q F_0$,
where $F_0$ is an automorphism of finite order of the rational vector 
space $\Lambda_{\bQ}$. Then $F_0$ normalizes $\Gamma$, and 
$$
\vert (H^0)^{\gamma F} \vert = 
\vert \Lambda/(\gamma F - \id)\Lambda \vert = 
\vert {\det}_{\Lambda_{\bQ}} (\gamma F - \id) \vert = 
{\det}_{\Lambda_{\bQ}} (q \id - F_0^{-1} \gamma^{-1})
$$
by results of \cite[3.2, 3.3]{Ca85}. As a consequence,
$$
\vert (H^0)^{\gamma F} \vert =
q^{\dim(H)} \; {\det}_{\Lambda_{\bQ}} (\id - F^{-1} \gamma^{-1}).
$$
Combined with (\ref{eqn:co1}), this yields
\begin{equation}\label{eqn:co2}
\vert (G/H)(\bF_q) \vert = 
\frac{\vert G^F \vert}{q^{\dim(H)}} \;
\frac{1}{\vert \Gamma \vert} \;
\sum_{\gamma \in \Gamma}  \frac{1}
{{\det}_{\Lambda_{\bQ}} (\id - F^{-1} \gamma)}.
\end{equation}
Now consider the expansion
$$
\frac{1}{{\det}_{\Lambda_{\bQ}} (\id - F^{-1} \gamma)} = 
\sum_{n=0}^{\infty} \Tr_{S^n(\Lambda_{\bQ})}(F^{-1} \gamma),
$$
where $S^n$ denotes the $n$th symmetric power. Since all eigenvalues
of $F$ in $\Lambda_{\bQ}$ have absolute value $q$, the series in the
right-hand side converges absolutely. Thus, we may write
\begin{equation}\label{eqn:co3}
\vert (G/H)(\bF_q) \vert = 
\frac{\vert G^F \vert} {q^{\dim(H)}} \sum_{n=0}^{\infty} 
\Tr_{S^n(\Lambda_{\bQ})}
\big(F^{-1}\; \frac{1}{\vert \Gamma\vert} 
\sum_{\gamma \in \Gamma} \gamma \big).
\end{equation}
Since the operator
$\frac{1}{\vert \Gamma \vert} \sum_{\gamma \in \Gamma} \gamma$
of any $\Gamma$-module $M$ is the projection onto the subspace
$M^{\Gamma}$ of $\Gamma$-invariants, the series in the right-hand side
of (\ref{eqn:co3}) equals
$$ 
\sum_{n=0}^{\infty} \Tr_{S^n(\Lambda_{\bQ})^{\Gamma}}(F^{-1}) =
\Tr_{S^{\Gamma}}(F^{-1}),
$$
where $S$ denotes the symmetric algebra of $\Lambda_{\bQ}$, and 
$S^{\Gamma}$ the subalgebra of $\Gamma$-invariants; here $F$ acts
on $S$ by algebra automorphisms, and preserves $S^{\Gamma}$. This
yields the equality 
\begin{equation}\label{eqn:co4}
\vert (G/H)(\bF_q) \vert = \frac{ \vert G^F \vert}{q^{\dim(H)}} \;
\Tr_{S^{\Gamma}}(F^{-1}).
\end{equation}
To obtain the desired combinatorial formula, it remains to compute 
$\vert G^F \vert$. For this, choose a maximal torus $T$ of $G$
containing $H^0$ (and defined over $k$), with normalizer $N_G(T)$ and
Weyl group $W$; denote by $\Lambda_T$ the character group of $T$, and
by $R$ its symmetric algebra over $\bQ$. Applying (\ref{eqn:co4}) to
$H = N_G(T)$ yields
$$
\vert G^F \vert  = \frac{q^{\dim(G)}}{\Tr_{R^W}(F^{-1})}
$$
in view of Steinberg's theorem. Substituting in (\ref{eqn:co4}) and
replacing $F$ with $F^n$ yields our combinatorial formula
\begin{equation}\label{eqn:co5}
\vert (G/H)(\bF_{q^n}) \vert = q^{n\dim(G/H)}
\frac{ \Tr_{S^{\Gamma}}(F^{-n})}{\Tr_{R^W}(F^{-n})}.
\end{equation}
Here $F$ acts on $R^W$ and $S^{\Gamma}$ by automorphisms of graded
algebras which are diagonalizable over $\bar{\bQ}$ with eigenvalues of
the form $\zeta \, q^j$, where $\zeta$ is a root of unity and $j$ is a
non-negative integer.

We now obtain an invariant-theoretical interpretation of the
right-hand side of (\ref{eqn:co5}). The restriction map
$\Lambda_T \to \Lambda_{H^0}$ induces a surjective, $F$-equivariant
homomorphism 
$$
\rho :  R \longrightarrow S.
$$
We claim that $\rho(R^W)$ is contained in $S^{\Gamma}$. To see this,
consider $\gamma \in \Gamma$ and its representative 
$h_{\gamma} \in H(\bar{k}) \subset N_G(H^0)(\bar{k})$. 
Then $T$ and $h_{\gamma}^{-1} T h_{\gamma}$ are maximal tori of the
centralizer $C_G(H^0)$. It follows that $h_{\gamma}$ is a
$\bar{k}$-rational point of $N_G(T) C_G(H^0)$. Thus, the automorphism
of $S$ induced by $\gamma$ lifts to an automorphism of $R$ induced by
some $w \in W$; this implies our claim.

By that claim, $S^{\Gamma}$ is a graded $R^W$-module; that module is
finitely generated, since the $R^W$-module $R$ is finitely generated. 
Moreover, $R^W$ is a graded polynomial algebra over $\bQ$, and hence 
$S^{\Gamma}$ admits a finite free resolution
$$
0 \to R^W \otimes E_m {\overset{\varphi_m}\longrightarrow}  
R^W \otimes E_{m-1} {\overset{\varphi_{m-1}}\longrightarrow}
\cdots {\overset{\varphi_1}\longrightarrow}
R^W \otimes E_0 \to S^{\Gamma} \to 0,
$$
where $E_0,\ldots,E_m$ are finite-dimensional vector spaces equipped
with an action of $F$, and $\varphi_1, \ldots, \varphi_m$ are 
$F$-equivariant homomorphisms of $R^W$-modules. Thus,  
$$
\Tr_{S^{\Gamma}}(F^{-n}) = \Tr_{R^W}(F^{-n}) \;
\sum_{i=0}^m (-1)^i \; \Tr_{E_i}(F^{-n})
$$
for all integers $n$. Together with (\ref{eqn:co5}), this yields
\begin{equation}\label{eqn:co6}
\vert (G/H)(\bF_{q^n}) \vert = q^{n \dim(G/H)} \sum_{i=0}^m
(-1)^i \; \Tr_{E_i} (F^{-n})
\end{equation}
for any positive integer $n$.

We may further assume that our free resolution is minimal,
i.e., each $\varphi_i$ maps bijectively a basis of $E_i$ to a 
minimal set of generators of the $R^W$-module 
$\Ima(\varphi_i) = \Ker(\varphi_{i-1})$ (see \cite[Lem.~19.4]{Ei95}). 
In particular, $E_0 \cong \bQ$, and each 
$E_i$ is $F$-equivariantly isomorphic to a subspace
of $R^W \otimes E_{i-1}$. Since the action of $F$ on $R^W$
is diagonalizable over $\bar{\bQ}$ with eigenvalues of the form 
$\zeta \, q^j$ as above, it follows by induction on $i$ that the same 
holds for the action of $F$ on $E_i$, and hence on $R^W \otimes E_i$. 
Together with (\ref{eqn:co6}), this yields that
$\vert (G/H)(\bF_{q^n}) \vert$ is a linear combination of powers 
$\zeta^{-n} \, q^{n(\dim(G/H) - j)}$ with integer coefficients. 
In other words, $\vert (G/H)(\bF_{q^n}) \vert$ is a periodic 
Laurent polynomial function of $q^n$ with algebraic integer 
coefficients. As noted at the beginning of Section \ref{sec:pepo}, 
this implies Theorem \ref{thm:pepo}.

We now adapt these arguments to prove Theorem \ref{thm:fact}.
Again, we may replace $\bF_q$ with $\bF_{q^N}$ and assume that the
torus $T$ is split, i.e., $F$ acts on $(\Lambda_T)_{\bQ}$ as 
$q \id$. Then the actions of $F$ on the algebras $R, R^W, S, \ldots$
are determined by their gradings, and 
$\Tr_R(F^{-n}), \Tr_{R^W}(F^{-n}), \Tr_S(F^{-n}), \ldots$ 
are obtained from the Hilbert series 
$h_R(t), h_{R^W}(t), h_S(t), \ldots$ 
of the corresponding graded algebras by putting $t = q^{-n}$. 

Consider the graded algebra    
$$
A := R \otimes_{R^W} S^{\Gamma}.
$$
Since $R$ is a free $R^W$-module of finite rank, the 
$S^{\Gamma}$-module $A$ is also free, of finite rank; moreover,
\begin{equation}\label{eqn:hil}
h_A(t) = \frac{ h_R(t) \; h_{S^{\Gamma}}(t)}{h_{R^W}(t)} =
\frac{h_{S^{\Gamma}}(t)}{(1-t)^{\dim(T)} \; h_{R^W}(t)}.
\end{equation}
The algebra of invariants $S^{\Gamma}$ is Cohen-Macaulay of Krull
dimension $\dim(H)$ (see e.g. \cite[Exercise 18.14]{Ei95}), 
and hence so is $A$. But $A$ is a finitely generated module over 
the polynomial algebra $R$. By Noether normalization, 
it follows that $A$ is a finitely generated module over a polynomial
subalgebra $R'\subset R$ generated by elements $z_1,\ldots,z_{\dim(H)}$
of degree $1$. In particular, $z_1,\ldots,z_{\dim(H)}$ generate an ideal 
of $A$ of finite codimension; therefore, they form a regular sequence
in $A$. It follows that $A$ is a free module of finite rank over $R'$
(see e.g. \cite[Cor.~18.17]{Ei95});
thus, the Hilbert series of $A$ satisfies
$$
h_A(t) = \frac{Q(t)}{(1-t)^{\dim(H)}},
$$ 
where $Q(t)$ is a polynomial with non-negative integer
coefficients. Combined with (\ref{eqn:hil}), this yields the equality
$$
\frac{h_{S^{\Gamma}}(t)}{h_{R^W}(t)} = (1-t)^r \; Q(t),
$$
where $r = \dim(T) - \dim(H)$. Evaluating at $q^{-n}$ and using 
(\ref{eqn:co5}) yields the desired factorization,
$$
\vert (G/H)(\bF_{q^n}) \vert = 
(q^n -1)^r \; q^{n(\dim(G/H)-r)} \; Q(q^{-n})
$$
(note that the function $(t-1)^r \; t^{\dim(G/H)-r} \; Q(t^{-1})$ 
is polynomial by Theorem \ref{thm:pepo}, and hence so is the function
$t^{\dim(G/H)-r} \; Q(t^{-1})$ ).

Finally, we prove Theorem \ref{thm:pos}. By (\ref{eqn:co0}), we may 
assume that $H$ is connected, and hence that
$$
\vert (G/H)(\bF_{q^n}) \vert = 
\frac{\vert G(\bF_{q^n}) \vert }{ \vert H(\bF_{q^n}) \vert }.
$$
By \cite{St68} (see also \cite[2.9]{Ca85}), the numerator and
denominator of the right-hand side are products of terms 
$q^{nd} - \zeta^n$, where $d$ is a positive integer, 
and $\zeta$ is either $0$ or a root of unity. It follows that the 
non-zero roots of the polynomials
$P_0(t), \ldots, P_{N-1}(t)$ of (\ref{eqn:pepo}) are roots of
unity. Since these polynomials have real coefficients, their roots are
either $0$, $1$, $-1$, or come by pairs of complex conjugates $\zeta$,
$\bar{\zeta}$, where $\zeta$ is a root of unity. Moreover, each
polynomial $(t + 1 - \zeta)(t + 1 -\bar{\zeta})$ has non-negative real 
coefficients; thus, the same holds for each $P_r(t)$.

\begin{remarks}
(i) The final step of the preceding proof does not extend to
all homogeneous varieties under linear algebraic groups, as
the polynomials $P_0(t), \ldots, P_N(t)$ may have roots of
absolute value $>1$. 

For example, let $\Gamma$ be a group of order $2$, and $H_r$ the 
semi-direct product of the split $r$-dimensional torus $\bG_m^r$ 
with $\Gamma$, where the non-trivial element $\gamma  \in \Gamma$ 
acts on $\bG_m^r$ via 
$(t_1, \ldots, t_r) \mapsto (t_1^{-1}, \ldots, t_r^{-1})$.
Then $H_r$ is isomorphic to a closed 
$F$-stable subgroup of $\GL(2r)$, where $F$ is the usual Frobenius 
endomorphism of $\GL(2r)$ that raises all matrix coefficients to 
the power $q$: denoting by $(e_1,\ldots,e_r,e_{-r}, \ldots, e_{-1})$ 
the standard basis of $k^{2r}$, we let 
$t = (t_1,\ldots,t_r) \in \bG_m^r = H_r^0$ act via
$t \cdot e_{\pm i} = t_i^{\pm 1} e_{\pm i}$ for $i = 1, \ldots,r$, 
and $\gamma$ act by exchanging $e_{\pm i}$ with $e_{\mp i}$.  
With the notation of the preceding proof, we have 
$$
\vert \GL(2r)^F \vert = q^{r(2r-1)} 
\prod_{i=1}^{2r} (q^i -1)
$$
and 
$$
\frac{1}{q^{\dim(H_r)}} \Tr_{S^{\Gamma}}(F^{-1}) 
= \frac{1}{2q^r}
\big( \frac{1}{(1 - q^{-1})^r} + \frac{1}{(1 + q^{-1})^r} \big)
= \frac{Q_r(q)}{(q^2 - 1)^r},
$$ 
where 
$$
Q_r(t) := \frac{1}{2}\big( (q + 1)^r + (q - 1)^r \big) =
\sum_{i, \, 0 \leq 2i \leq r} {r \choose 2 i} t^{r -2i}.
$$
By (\ref{eqn:co4}), it follows that the homogeneous variety 
$X_r := \GL(2r)/H_r$ satisfies
$$
\vert X_r(\bF_q) \vert = Q_r(q) \, q^{r(2r - 1)} \,
\prod_{i=1}^r (q^{2i -1} - 1) \,
\prod_{j=1}^r \frac{q^{2j} - 1}{q^2 - 1} .
$$
In particular, $\vert X_r(\bF_q) \vert$ is a polynomial in $q$,
and the maximal absolute value of its roots tends to infinity
as $r \to \infty$. 

\medskip

\noindent
(ii) By the reduction steps in Section \ref{sec:fact} and the equation
(\ref{eqn:co2}), we may take for the period $N$ the least common 
multiple of the orders of all the Frobenius endomorphisms of tori 
with dimension $\leq \rk(G)$. As a consequence, we may take 
$$
N = \lcm\big(n, \varphi(n) \leq \rk(G)\big),
$$
where $\varphi$ denotes the Euler function.   

\end{remarks}

\end{document}